\documentclass[leqno, 12pt]{article}
\usepackage{amsmath,amsfonts,amsthm,amssymb,indentfirst}

\newtheorem{theorem}{Theorem}
\newtheorem{lemma}[theorem]{Lemma}
\newtheorem{definition}[theorem]{Definition}
\newtheorem{corollary}[theorem]{Corollary}
\newtheorem{proposition}[theorem]{Proposition}

\theoremstyle{definition}
\newtheorem{example}[theorem]{Example}

\newcommand{\End}{\mathrm{End}}

\newcommand{\ann}{\mathrm{ann}}

\newcommand{\rad}{\mathrm{rad}}
\newcommand{\Nil}{\mathrm{Nil}^*}

\newcommand{\Q}{\mathbb{Q}}
\newcommand{\Z}{\mathbb{Z}}

\newcommand{\M}{\mathbb{M}}

\begin{document}

\title{The ideals of an ideal extension}

\author{Zachary Mesyan\thanks{This work was done while the author was supported by a Postdoctoral Fellowship from the Center for Advanced Studies in Mathematics at Ben Gurion University, a Vatat fellowship from the Israeli Council for Higher Education, and ISF grant 888/07.}}

\maketitle

\begin{abstract}
Given two unital associative rings $R \subseteq S$, the ring $S$ is said to be an {\em ideal} (or {\em Dorroh}) {\em extension} of $R$ if $S = R \oplus I$, for some ideal $I \subseteq S$. In this note we investigate the ideal structure of an arbitrary ideal extension of an arbitrary ring $R$. In particular, we describe the Jacobson and upper nil radicals of such a ring, in terms of the Jacobson and upper nil radicals of $R$, and we determine when such a ring is prime and when it is semiprime. We also classify all the prime and maximal ideals of an ideal extension $S$ of $R$, under certain assumptions on the ideal $I$. These are generalizations of earlier results in the literature.
\end{abstract}

\section{Introduction}

Throughout this note ``ring" will mean a unital associative ring, and ``rng" will refer to an associative ring that may not possess a unit. Given a ring $R$ and a rng $I$, we will say that $I$ is an $R$-{\em rng} if it is an $(R,R)$-bimodule, for which the actions of $R$ are compatible with multiplication in $I$ (i.e., $r(ij) = (ri)j$, $i(rj) = (ir)j$, and $(ij)r = i(jr)$ for every $r \in R$ and $i,j \in I$). If $R$ is a ring and $I$ an $R$-rng, then one can turn the abelian group $R \oplus I$ into a ring by defining multiplication by $(r,i) \cdot (p,j) = (rp, ip + rj + ij)$ for $r,p \in R$ and $i,j \in I$. Such an ring is called an {\em ideal extension} (it is also called a {\em Dorroh extension}), and we will denote it by $E(R,I)$. It is easy to verify that what we called an ideal extension in the abstract is isomorphic to a ring constructed as above. Dorroh \cite{Dorroh} first used this construction, with $R = \Z$ (the ring of integers), as a means of embedding a (nonunital) rng $I$ into a (unital) ring. However, such extensions $E(R,I)$ have proved to be useful in a number of other situations. For instance, Nicholson and Zhou \cite{NZ1} have used them to construct uniquely clean rings, i.e., ones where every element can be written as a sum of a unit and an idempotent in exactly one way (cf.\ also \cite{NZ2}). Ideal extensions have also played a very important role in classifying the minimal ring extensions of an arbitrary prime ring (cf.\ \cite{DM}). They have received particular attention in the case where $I^2 = 0$ (cf.\ \cite{Nagata}). Several different names have been used in the literature to refer to $E(R,I)$ in this situation, specifically {\em trivial extension}, {\em idealization}, and {\em split-null extension}. More general versions of the above construction have been studied as well. For instance, given two rngs $I$ and $J$, Everett \cite{Everett} described all rngs $R$ such that $J$ is an ideal of $R$ and $R/J = I$ (cf.\ also \cite{Petrich} and \cite{RB}).

In this note we study the ideal structure of an arbitrary ring of the form $E(R,I)$. Various results on this subject have appeared before; our goal here is to extend them and collect them in one place. We first describe all (two-sided) ideals, as well as all nilpotent and nil ideals of such a ring (Section~\ref{idealsection}). Then, in Section~\ref{radicalsection} we show that an element $(r,i) \in E(R,I)$ belongs to $\rad(E(R,I))$, the Jacobson radical of $E(R,I)$, if and only if $r \in \rad(R)$ and $jr+ji \in \rad(I)$ for all $j \in I$. This generalizes the well-known fact that $\rad(E(\Z,I)) = 0 \oplus \rad(I)$, as well as a theorem of Haimo \cite{Haimo} for certain commutative rings. We give an analogous description of the upper nil radical of $E(R,I)$ in Section~\ref{nilradsection}.

In Sections~\ref{semiprimesection} and~\ref{primesection} we determine when $E(R,I)$ is semiprime and when it is prime, generalizing results from \cite{DM}. More specifically, $E(R,I)$ is prime if and only if $I$ is prime, $\ann_R(I) = 0$, and there do not exist a nonzero $R$-subrng $J \subseteq I$ and an $R$-homomorphism $\varphi : J \rightarrow R$ such that for all $i \in I$, $j \in J$ one has $ij = i\varphi(j)$, $ji = \varphi(j)i$. Our description of the semiprime rings of the form $E(R,I)$ is similar. In Section~\ref{primeidealsection} we describe all the prime and maximal ideals of $E(R,I)$ in the case where $I$ comes equipped with an $R$-homomorphism $\varphi : I \rightarrow R$ that satisfies $i\varphi(j) = ij = \varphi(i)j$ for all $i,j \in I$ (for instance, every ideal $I$ of $R$ satisfies this property). More specifically, we show that in this situation an ideal $K \subseteq E(R, I)$ is prime (respectively, maximal) if and only if either $K = A \oplus I$ for some prime (respectively, maximal) ideal $A$ of $R$, or $K = \{(a, -i) : i \in I, a \in R, \text{ such that } a - \varphi(i) \in Z\}$, where $Z$ is a prime (respectively, maximal) ideal of $R$ and $\varphi (I) \not\subseteq Z$. This generalizes a theorem of D'Anna and Fontana \cite{DF} for commutative rings (cf.\ also \cite{DF2} and \cite{MY}). (In these three articles an ideal extension is called an {\em amalgamated duplication of a ring along an ideal}.)

While our primary interest is in two-sided ideals, we briefly discuss left ideals of $E(R,I)$ in Section~\ref{leftidealsection}. In particular, we describe all left ideals of such a ring and then use this description to determine when $E(R,I)$ is left noetherian and when it is left artinian. One-sided ideals of $E(R,I)$, in the case where $R$ is a commutative ring and $I$ is an $R$-algebra, are also discussed by Birkenmeier and Heatherly in \cite{BH}. More specifically, they give a necessary and sufficient condition for such $E(R,I)$ to be a strongly right bounded ring, i.e., one in which every nonzero right ideal contains a nonzero ideal.

\subsection*{Acknowledgements}  

The author is grateful to Tom Dorsey for helpful discussions about this material, Jay Shapiro for references to related literature and suggestions about directions to explore, George Bergman for comments on an earlier draft of this note, and the referee for suggestions that have led to numerous improvements in the paper.

\section{Definitions}

We begin by collecting a few basic definitions that will be needed throughout the article. All modules and bimodules over a ring will be assumed to be unital.

\begin{definition}
Let $R$ be a ring, and let $I$ be a rng. We say that $I$ is an $R$-{\em rng} if $I$ is an $(R,R)$-bimodule, and for all $r \in R$ and $i,j \in I$ one has $r(ij) = (ri)j$, $i(rj) = (ir)j$, and $(ij)r = i(jr)$.  

Given two $R$-rngs $I$ and $J$, we say that $\varphi: I \rightarrow J$ is an {\em $R$-homomorphism} if it is a rng homomorphism that is also an $(R,R)$-bimodule homomorphism.
\end{definition}

For instance, any ideal $I$ of a ring $R$ is an $R$-rng. Also, any homomorphism $R \rightarrow S$ of rings equips $S$ with the structure of an $R$-rng. One can also turn any $(R,R)$-bimodule $I$ into an $R$-rng by declaring the product of any two elements of $I$ to be zero.

Given an $R$-rng $I$, let $\ann_R(I) = \{r \in R: rI = Ir = 0\}$.  This annihilator is an ideal of $R$. (Throughout this note, ``ideal," if not
modified by ``right'' or ``left,'' will mean a two-sided ideal.)

\begin{definition} \label{defidealext}
Given a ring $R$ and an $R$-rng $I$, the {\em ideal extension} $($also known as the {\em Dorroh extension}$)$ $E(R,I)$ is the object that has the abelian group structure of $R \oplus I$ and multiplication given by $$(r,i) \cdot (p,j) = (rp, ip + rj + ij),$$ for $r, p \in R$ and $i, j \in I$.
\end{definition}
 
It is straightforward to verify that $E(R,I)$ is a ring, where $R \oplus 0$ is a subring (having the same unit) and $0 \oplus I$ is an ideal. We will abuse notation by identifying $R \oplus 0 \subseteq E(R,I)$ with $R$ and $0 \oplus I \subseteq E(R,I)$ with $I$, whenever convenient. 

\begin{definition} \label{centraldef}
We say that an $R$-rng $I$ is {\em centrally generated} if $I$ is generated as a $($left\,$)$ $R$-module by elements that commute with all elements of $R$.
\end{definition}

\begin{definition}
Given an $R$-rng $I$, we say that $J \subseteq I$ is an $R$-{\em ideal} of $I$ if $J$ is an $(R,R)$-subbimodule that is also an ideal in the rng $I$.
\end{definition}

In general, an ideal of an $R$-rng $I$ need not be an $R$-ideal. For instance, if $I^2 = 0$, then for any $i \in I$, the set $\{ai : a \in \Z\}$ is an ideal of $I$. However, such a set certainly need not be closed under multiplication by $R$.

\section{Arbitrary ideals} \label{idealsection}

Our first result describes an arbitrary ideal in a ring of the form $E(R,I)$. This is a generalization to arbitrary $R$-rngs of \cite[Proposition 3.1]{DM}, and it will be used throughout this note. A similar description of an arbitrary left ideal of $E(R,I)$ is given in Proposition~\ref{leftidealdescription}. The statement of the result is a bit technical, but, vaguely speaking, it says that every ideal of $E(R,I)$ is determined by some $R$-subrng $J\subseteq I$, an ideal $Z \subseteq R$, and a ``well-behaved" $R$-homorphism $\varphi: J \rightarrow R/Z$.

\begin{proposition}  \label{idealdescription} Let $R$ be a ring, and let $I$ be an $R$-rng.  Then every ideal of $E(R,I)$ must be of the form $$K = \{(a, -j): a \in A, j \in J, \text{ such that } a+Z = \varphi(j)\},$$ where $Z \subseteq A$ are ideals of $R$,  $J$ is an $R$-subrng of $I$, and $\varphi: J \rightarrow A/Z$ is a surjective $R$-homomorphism such that for all $(a,-j) \in K$ and $i \in I$ the following are satisfied
\begin{enumerate}
\item[$($a$)$] $ai-ji \in \ker(\varphi),$
\item[$($b$)$] $ia-ij \in \ker(\varphi).$
\end{enumerate}  
Conversely, given $A$, $J$, $Z$, and $\varphi$ as above, $K = \{(a, -j): a \in A, j \in J, \text{ such that } a+Z = \varphi(j)\}$ is an ideal of $E(R,I)$.

Further, given an ideal $K \subseteq E(R,I)$, with appropriate $A$, $J$, $Z$, and $\varphi$, the following hold.
\begin{enumerate}
\item[$(1)$] $J$ is an $R$-ideal of $I$ if and only if $AI + IA \subseteq J$. In particular, if $A = 0$, then $J$ is an $R$-ideal of $I$, and if $J = 0$, then $A \subseteq \ann_R(I)$.
\item[$(2)$] $A = Z$ if and only if $\varphi = 0$ if and only if $K = A \oplus J$. In this case $J$ is an $R$-ideal of $I$.
\item[$(3)$] If $\varphi$ is injective, then $Z \subseteq \ann_R (I)$.
\item[$(4)$] $Z \oplus \ker(\varphi) \subseteq K$ is an ideal of $E(R,I)$. In particular, $\, \ker(\varphi)$ is an $R$-ideal of $I$.
\end{enumerate}
\end{proposition}

\begin{proof}
Suppose that $K$ is an ideal of $E(R,I)$, and let $A \subseteq R$ be the set of first coordinates of elements of $K$, namely $\{r \in R : (r, i) \in K \text{ for some } i \in I\}$. It is easy to see that $A$ must be an ideal of $R$. Now consider the set $J$ of all second coordinates of elements of $K$, that is $\{i \in I : (r,i) \in K \text{ for some } r \in R\}$.  It is clear that $J$ is an $(R,R)$-subbimodule of $I$. Letting $(r,i),(r',i') \in K$ be any two elements, the equation $$(r,i)(r',i') - (r,i)r' - r(r',i') = (-rr', ii')$$ shows that $J$ is an $R$-subrng of $I$. Setting $Z = \{r \in R: (r,0) \in K\}$, it is clear that $Z$ is an ideal of $R$ contained in $A$. We next define a map $\varphi: J \rightarrow A/Z$ as follows.  Given $j \in J$, we can find some $a \in A$ such that $(a,-j) \in K$.  Let $\varphi(j)$ be the image of $a$ in $A/Z$ under the natural projection $A \rightarrow A/Z$.  To see that $\varphi$ is well-defined, it suffices to note that if $(a,-j),(a',-j) \in K$, then $(a-a',0) \in K$, and hence $a-a' \in Z$.  To see that $\varphi$ is an $R$-homomorphism, we note that if $(a,-j),(a',-j') \in K$, then $(a+a',-(j+j')),(aa',-aj'), (aa',-ja')$, and $(aa',-jj')$ are all elements of $K$ (the last established by the equality above). The membership of the first three in $K$ shows (reducing modulo $Z$) that $\varphi$ is an $(R,R)$-bimodule homomorphism, while that of the last shows (reducing modulo $Z$) that $\varphi(jj') = \varphi(j)\varphi(j')$.  Thus $\varphi$ is an $R$-homomorphism, and it is clearly surjective. We can now check that $K = \{(a,-j): a \in A, j \in J, a + Z = \varphi(j)\}$.  For each $j \in J$, we can find some $a \in A$ such that $(a,-j) \in K$.  But $Z \oplus 0 \subseteq K$,
and hence $(a+Z,-j) \subseteq K$.  It follows that $a+Z = \varphi(j)$, and hence $\{(a,-j): a \in A, j \in J, a + Z = \varphi(j)\} \subseteq K$. That the reverse inclusion also holds is clear from the definitions of $A$, $J$, $Z$, and $\varphi$.  It remains to check that $\varphi$ satisfies conditions (a) and (b). This follows from the fact that if $(a,-j) \in K$ and $i \in I$, then $(a,-j)(0,i) = (0,ai - ji)$ and $(0,i)(a,-j) = (0,ia-ij)$ are elements of $K$.

For the converse, let $K = \{(a, -j): a \in A, j \in J, a+Z = \varphi(j)\}$, for appropriate $A$, $J$, $Z$, and $\varphi$. Since $\varphi$ is an $R$-homomorphism, it is straightforward to see that $K$ is an $(R,R)$-subbimodule of $E(R,I)$. Since $R+I = E(R,I)$, to finish showing that $K$ is an ideal of $E(R,I)$, it will suffice to check that $IK$ and $KI$ are contained in $K$.  Thus, let $i \in I$ and $(a,-j) \in K$ be any elements.  Then $(a,-j)(0,i) = (0,ai -ji) \in K,$ by condition (a), showing that $KI \subseteq K$. Condition (b) similarly implies that $IK \subseteq K$, allowing us to conclude that $K$ is an ideal of $E(R,I)$.  

To prove (1), suppose that $J$ is an $R$-ideal, in addition to being an $R$-subrng of $I$. Then for all $j \in J$ and $i \in I$, we have $ij, ji \in J$. Conditions (a) and (b) then imply that $AI + IA \subseteq J$. Conversely, if $ai, ia \in J$ for all $i \in I$ and $a \in A$, then conditions (a) and (b) imply that $ji, ij \in J$ for all $j \in J$ and $i \in I$. Thus, $J$ is an $R$-ideal if and only if $AI + IA \subseteq J$. The remaining claims in (1) are immediate.

In (2), the equivalence of  $A = Z$ and $\varphi = 0$ is clear. Further, $A = Z$ if and only if for all $j \in J$, we have $(0, j) \in K$. It follows that $A = Z$ if and only if $K = A \oplus J$.  Now, if $K = A \oplus J$, then for any $j \in J$ and $i \in I$ both $(0, ji) = (0,j)(0,i)$ and $(0,ij) = (0,i)(0,j)$ must be elements of $K$, since it is an ideal.  It follows that $J$ is an $R$-ideal of $I$. 

For (3), suppose that $\varphi$ is injective, and therefore that $\ker(\varphi) = 0$. Now, for all $z \in Z$ we have $(z,0) \in K$, and hence $zi - 0i, iz - i0 \in \ker(\varphi)$ for all $i \in I$. Thus, $ZI = IZ = 0$, showing that $Z \subseteq \ann_R(I)$.

Finally, (4) follows from the main statement of the proposition, since $$Z \oplus \ker(\varphi) = \{(z, -j): z \in Z, j \in \ker(\varphi), z + Z = \varphi(j)\},$$ and for all $z \in Z$, $j \in \ker(\varphi)$, and $i \in I$, we have $zi - ji, iz - ij \in \ker(\varphi)$ (since $(z, -j) \in K$). The last claim follows from (2).
\end{proof}

If the $\varphi$ in the above statement happens to extend to an $R$-homomorphism $I \rightarrow R/Z$, then the conditions (a) and (b) are satisfied automatically, as the following corollary shows.

\begin{corollary}\label{idealdescription2}
Let $R$ be a ring, let $I$ be an $R$-rng, and let $K = \{(a, -j): a \in R, j \in J, \text{ such that } a+Z = \varphi(j)\}$, where $Z \subseteq R$ is an ideal,  $J$ is an $R$-subrng of $I$, and $\varphi: I \rightarrow R/Z$ is an $R$-homomorphism. Then $K$ is an ideal of $E(R,I)$.
\end{corollary}

\begin{proof}
It suffices to show that the conditions (a) and (b) of Proposition~\ref{idealdescription} are satisfied for all $(a,-j) \in K$ and $i \in I$. Thus, letting $(a,-j)$ and $i$ be such elements, we have $(a,-j)(0,i) = (0,ai - ji)$.  Since $\varphi$ is an $R$-homomorphism, $\varphi(ai - ji) = a \varphi(i) - \varphi(j)\varphi(i)$.  But, by assumption, $\varphi(j) = a + Z$, and so $\varphi(ai - ji) = 0 + Z$. Thus, $ai - ji \in \ker(\varphi)$, which establishes (a). Condition (b) can be verified similarly. 
\end{proof}

Statements (1), (2), and (3) in Proposition~\ref{idealdescription}, as well as Corollary~\ref{idealdescription2}, describe situations in which an ideal of $E(R,I)$ can have a particularly nice form.  However, in general, an ideal of $E(R,I)$ need not have any such form, as the following example shows.

\begin{example}
Let $E = E(\Z, \Q)$ and $K = \{(a, -a) : a \in \Z\}$ (where $\Q$ is the ring of the rational numbers). By Proposition~\ref{idealdescription}, $K$ is an ideal of $E$ (with $J = A = \Z$, $Z = 0$, and $\varphi$ the identity map). But, $J = \Z$ is not a ($\Z$-)ideal of $\Q$, and $\varphi$ does not extend to a ($\Z$-)homomorphism $\Q \rightarrow \Z$.

One can make this example more interesting by letting $L$ be any rng and taking $R = A = \Z \times \Z$, $I = \Q \times L$ (with the product action of $\Z \times \Z$), $J = \Z \times L$, $Z = 0 \times \Z$, and $\varphi: J \rightarrow R/Z$ defined by $\varphi(a,b) = (a,0) + Z$ for all $(a,b) \in J$. Here $Z \neq 0$, $\varphi$ is not injective, and $Z \not\subseteq \ann_R(I)$.
\end{example}

Let us next describe the nilpotent and nil ideals of $E(R,I)$.

\begin{lemma} \label{nilideals}
Let $R$ be a ring, let $I$ be an $R$-rng, and let $K = \{(a, -j): a \in A, j \in J, a+Z = \varphi(j)\}$ be an ideal of $E(R,I)$, for appropriate $A$, $J$, $Z$, and $\varphi$ $($cf.\ Proposition~\ref{idealdescription}\,$)$. Then the following are equivalent:
\begin{enumerate}
\item[$(1)$] $K$ is nilpotent $($respectively, nil\,$)$,
\item[$(2)$] $A$ and $J$ are nilpotent $($respectively, nil\,$)$,
\item[$(3)$] $A$ and $\, \ker(\varphi)$ are nilpotent $($respectively, nil\,$)$.
\end{enumerate}
\end{lemma}

\begin{proof}
We will prove only the ``nilpotent" version of the statement, since the ``nil" version can be proved similarly.

Suppose that $K^n = 0$ for some positive integer $n$. Since $Z \oplus \ker(\varphi) \subseteq K$, we have $Z^n = 0$ and $\ker(\varphi)^n = 0$. If $a_1, \dots, a_n \in A$ are any elements, then $(a_1, - j_1), \dots, (a_n, - j_n) \in K$ for some $j_1, \dots, j_n \in J$. Since $(a_1, - j_1) \dots (a_n, - j_n) = 0$, we see that $a_1 \dots a_n = 0$, and hence $A^n = 0$. Now $J/\ker(\varphi)$ is $R$-isomorphic to $A/Z$, and hence $J^n \subseteq \ker(\varphi)$. But, $\ker(\varphi)^n = 0$, and therefore $J^{n^2} = 0$. In particular, both $A$ and $J$ are nilpotent. Thus, (1) implies (2), and, of course, (2) implies (3) tautologically.

To show that (3) implies (1), suppose that $A^n = 0$ and $\ker(\varphi)^m = 0$ for some positive integers $n$ and $m$. Let $(a_1, -j_1), (a_2, -j_2), \dots, (a_{mn}, -j_{mn}) \in K$ be any elements. Now, $$(a_1, -j_1) (a_2, -j_2) \dots (a_{mn}, -j_{mn}) = ((a_1, -j_1) \dots (a_{n}, -j_{n}))((a_{n+1}, -j_{n+1}) \dots (a_{2n}, -j_{2n}))$$ $$\dots ((a_{(m-1)n+1}, -j_{(m-1)n+1}) \dots (a_{mn}, -j_{mn})).$$ Since $A^n = 0$, this is a product of $m$ elements of $K \cap (0 \oplus J) = 0 \oplus \ker(\varphi)$. Since $\ker(\varphi)^m = 0$, the entire product must be zero, showing that $K^{mn} = 0$.
\end{proof}

We finish the section with some easy observations to which it will be convenient to refer in the sequel.

\begin{lemma}\label{dirsumideal}
Let $R$ be a ring, and let $I$ be an $R$-rng. 
\begin{enumerate}
\item[$(1)$] If $J$ is an $R$-ideal of $I$, then $0 \oplus J$ is an ideal of $E(R,I)$. 
\item[$(2)$] If $A \subseteq \ann_R(I)$ is an ideal of $R$, then $A \oplus 0$ is an ideal of $E(R,I)$.
\item[$(3)$] Let $A$ be an ideal of $R$, and let $J$ be an $R$-ideal of $I$. If $AI + IA \subseteq J$, then $A \oplus J$ is an ideal of $E(R,I)$.
\end{enumerate}
\end{lemma}

\begin{proof}
All three claims can be quickly verified directly or deduced from Proposition~\ref{idealdescription}.
\end{proof}

By Proposition~\ref{idealdescription}, the converse of statement (3) in the above lemma also holds. 

\begin{lemma}\label{centrallemma}
Let $R$ be a ring, and let $I$ be a centrally generated $R$-rng. If $A$ is an ideal of $R$, then $AI = IA$, and it is an $R$-ideal of $I$. 
\end{lemma}

\begin{proof}
Write $I =  \sum_{l} Rx_l$, where the elements $x_l \in I$ commute with all elements of $R$. Then $AI = \sum_{l} Ax_l = \sum_{l} x_lA = IA$. Now, it is clear that $AI = IA$ is an $(R,R)$-bimodule, that $AI$ is closed under multiplication on the right by elements of $I$, and that $IA$ is closed under multiplication on the left by elements of $I$. Hence, $AI = IA$ is an $R$-ideal.
\end{proof}

\section{The Jacobson radical} \label{radicalsection}

Let us remind the reader that an element $i$ of a rng $I$ is said to be {\em left quasi-regular} if $i + k + ki = 0$ for some $k \in I$, and it is said to be {\em right quasi-regular} if $i + k + ik = 0$ for some $k \in I$. The {\em Jacobson radical of $I$}, denoted $\rad(I)$, is the set $$\{i \in I : ji \text{ is left quasi-regular for all } j \in I \}$$ $$ = \{i \in I : ij \text{ is right quasi-regular for all } j \in I \},$$ which is an ideal of $I$. It is not hard to show that all elements of $\rad(I)$ are left and right quasi-regular (e.g., cf.~\cite[Chapter 4]{Lam}). If $I$ is a ring, then $\rad(I)$ $$= \{i \in I : \forall j \in I, \ \exists k \in I \text{ such that } k(1-ji) = 1\}$$ $$ = \{i \in I : \forall j \in I, \ \exists k \in I \text{ such that } (1-ij)k = 1\}.$$ It is well known that for a rng $I$, $\rad(E(\Z, I)) = 0 \oplus \rad(I)$ (e.g., cf.\ \cite[Chapter 4, Lemma 63]{Divinsky}, also cf.\ \cite{RH} and \cite{Feigelstock} for generalizations of this to arbitrary radicals in place of the Jacobson radical). The ideal $\rad(E(R,I))$, for certain commutative rings $E(R,I)$, is described in \cite[Theorem 6]{Haimo} (more on this below). The goal of this section is to describe $\rad(E(R,I))$ for all rings $R$ and $R$-rngs $I$. We begin with a couple of lemmas.

\begin{lemma}
Let $R$ be a ring, and let $I$ be an $R$-rng. Then $\, \rad(I)$ is an $R$-ideal of $I$.
\end{lemma}

\begin{proof}
Since $\rad(I)$ is an ideal of $I$, we just need to show that it is closed under multiplication by $R$. Let $i \in \rad(I)$, $j \in I$, and $r \in R$ be any elements. Then we can find $k \in I$ such that $0 = (jr)i + k + k(jr)i = j(ri) + k + kj(ri)$, and hence $RI \subseteq \rad(I)$. Using the description of $\rad(I)$ in terms of right quasi-regular elements, one can similarly show that $IR \subseteq \rad(I)$.
\end{proof}

\begin{lemma}\label{radI}
Let $R$ be a ring, and let $I$ be an $R$-rng. Then $I \cap \rad(E(R,I)) = \rad(I)$.
\end{lemma}

\begin{proof}
Let $i \in \rad(I)$ and $(r, j) \in E(R,I)$ be any elements. Then $(r, j)(0, i) = (0, ri + ji) \in 0 \oplus \rad(I)$, by the previous lemma. Hence, there exists $k \in I$ such that $(ri + ji) + k + k(ri + ji) = 0$. It follows that $(r, j)(0, i) + (0, k) + (0, k)(r, j)(0, i) = 0.$ This shows that $0 \oplus \rad(I) \subseteq \rad(E(R,I))$. 

For the opposite inclusion, suppose that $(0, i) \in I \cap \rad(E(R,I))$, and let $j \in I$ be any element. Then there exists $(r, k) \in E(R,I)$ such that $$0 = (0, j)(0, i) + (r, k) + (r, k)(0, j)(0, i) = (0, ji) + (r, k) + (0, rji + kji).$$ It follows that $r = 0$, and hence $0 = ji + k + kji$, showing that $i \in \rad(I)$.
\end{proof}

We are now ready to describe $\rad(E(R,I))$.

\begin{theorem} \label{radE} 
Let $R$ be a ring, let $I$ be an $R$-rng, and let $r \in R$, $i \in I$ be any elements. Then the following are equivalent:
\begin{enumerate}
\item[$(1)$] $(r, i) \in \rad(E(R,I))$,
\item[$(2)$] $r \in \rad(R)$ and $jr+ji \in \rad(I)$ for all $j \in I$,
\item[$(3)$] $r \in \rad(R)$ and $rj+ij \in \rad(I)$ for all $j \in I$.
\end{enumerate}
Moreover, $(r, 0) \in \rad(E(R,I))$ if and only if $r \in \{r \in \rad(R) : rI + Ir \subseteq \rad(I)\}$, and $(0, i) \in \rad(E(R,I))$ if and only if $i \in \rad(I)$.
\end{theorem}

\begin{proof}
Suppose that $(r, i) \in \rad(E(R,I))$, and let $p \in R$ and $j, j' \in I$ be any elements.  Then $1 - (p, 0)(r, i) = (1 - pr, -pi)$ and $1 - (0, -j'j)(r, i) = (1, j'jr + j'ji)$ must be left-invertible. That is, $$(1, 0) = (s, k)(1 - pr, -pi) = (s(1 - pr), -spi + k(1 - pr) - kpi)$$ and $$(1, 0) = (t, l)(1, j'jr+j'ji) = (t, t(j'jr + j'ji) + l + l(j'jr + j'ji))$$ for some $(s, k), (t, l) \in E(R, I)$. In particular, $1 = s(1 - pr) \in R$, showing that $r \in \rad(R)$. Also, we must have $t = 1$, and therefore $0 = j'(jr+ji) + l + lj'(jr+ji)$. Thus, $jr+ji \in \rad(I)$ for all $j \in I$. Using the characterization of the Jacobson radical in terms of right-invertible elements, one can similarly show that $rj+ij \in \rad(I)$ for all $j \in I$. Therefore, (1) implies (2) and (3).

Now, suppose that $r \in \rad(R)$ and $i \in I$ are such that $jr+ji \in \rad(I)$ for all $j \in I$. Let $(p, j) \in E(R,I)$ be any element. To prove that (2) implies (1), we will show that $1 - (p,j)(r,i)$ is left-invertible. Since $r \in \rad(R)$, we can find $s \in R$ such that $1 = s(1-pr)$. Now, $$(s,spi)(1 - (p,j)(r,i)) = (s,spi)(1 - pr, - jr - pi - ji)$$ $$= (s(1-pr), -s(jr + pi + ji) + spi - spipr - spi(jr + pi + ji))$$ $$= (1, - s(jr + ji) -spi(jr + ji) - sp(ipr + ipi)) \in 1 + 0 \oplus \rad(I).$$ But, by the previous lemma, $0 \oplus \rad(I) \subseteq \rad(E(R,I))$, and hence an element of $1 + 0 \oplus \rad(I)$ must be left-invertible. It follows that $1 - (p,j)(r,i)$ is left-invertible as well.

That (1) follows from (3) can be shown by using a similar computation to deduce that in this case $(1-(r,i)(p,j))(s,ips) \in 1 + 0 \oplus \rad(I)$, and hence that  $1-(r,i)(p,j)$ is right-invertible. The final claim follows from the equivalence of (1), (2), and (3), and from the previous lemma.
\end{proof}

As an immediate consequence, we obtain the following.

\begin{corollary} \label{radcor}
Let $R$ be a ring, and let $I$ be an $R$-rng. Then the following are equivalent:
\begin{enumerate}
\item[$(1)$] $\rad(E(R,I)) = \rad(R) \oplus \rad(I)$,
\item[$(2)$] $I\rad(R) \subseteq \rad(I)$,
\item[$(3)$] $\rad(R)I \subseteq \rad(I).$
\end{enumerate}
In particular, if $\, \rad(R) = 0$, then $\, \rad(E(R,I)) = 0 \oplus \rad(I)$.
\end{corollary}

Written in our notation, \cite[Theorem 6]{Haimo} says that if $R$ is a commutative ring and $I$ is an $R$-algebra such that $r(i - i^2) = 0$ for all $r \in \rad(R)$ and $i \in I$, then $\rad(E(R,I)) = \rad(R) \oplus \rad(I)$. This can be quickly deduced from the above corollary, since in this situation $\rad(R)I \subseteq \rad(I)$. For, given any $r \in \rad(R)$ and $i,j \in I$, choosing $p \in R$ so that $r + p + pr = 0$, one has $$j(ri) + jpi + (jpi)j(ri) = rji + pji + pr(ji)^2$$ $$= rji + pji + prji = (r + p + pr)ji = 0,$$ and hence $ri \in \rad(I)$.

\begin{corollary}
Let $R$ be a ring, and let $I$ be an $R$-rng that is finitely generated as a $($left\,$)$ $R$-module by elements that commute with all elements of $R$. Then $\, \rad(E(R,I)) = \rad(R) \oplus \rad(I)$.
\end{corollary}

\begin{proof}
We begin by recalling a standard fact from ring theory. Let $R$ and $S$ be two rings, and let $f : R \rightarrow S$ be a ring homomorphism, such that $S = f(R)x_1 + f(R)x_2  + \dots + f(R)x_n$ for some $x_1, x_2, \dots, x_n \in S$ that commute with all elements of $f(R)$. Then $f(\rad(R)) \subseteq \rad(S)$ (cf.\ \cite[Proposition 5.7]{Lam}). 

Now, let $R$ and $I$ be as in the statement, set $S = E(R,I)$, and let $f : R \rightarrow S$ be the natural inclusion. By hypothesis, we can write $S = R \oplus (Rx_1 + Rx_2 + \dots + Rx_n)$ for some $x_1, x_2, \dots, x_n \in I$ that commute with all elements of $R$. Hence, by the previous paragraph, $\rad(R) \subseteq \rad(E(R,I))$. Finally, by Theorem~\ref{radE}, $\rad(E(R,I)) = \rad(R) \oplus \rad(I)$.
\end{proof}

We finish the section with two examples of rings $E(R, I)$ where $\rad(E(R,I)) \neq \rad(R) \oplus \rad(I)$.

\begin{example} \label{matrixeg}
Let $F$ be a field, let $R = \mathbb{T}_2(F)$, the ring of upper-triangular $2 \times 2$ matrices over $F$, and let $I = \M_2(F)$, the ring of all $2 \times 2$ matrices over $F$. Set $$K = \Big\{ \Big(\Big(\begin{array}{ll}
0 & a \\
0 & 0
\end{array}\Big),
\Big(\begin{array}{rr}
0 & -a \\
0 & 0
\end{array}\Big)\Big) : a \in F \Big\}.$$
It is easy to see that $K$ is an ideal of $E(R,I)$ and that $K^2 = 0$. (This can be done directly, or by using Proposition~\ref{idealdescription}, with $$A = J =  \Big\{\Big(\begin{array}{ll}
0 & a \\
0 & 0
\end{array}\Big) : a \in F \Big\},$$ $Z = 0$, and $\varphi$ the identity map.) Since $$\rad(R) =  \Big\{\Big(\begin{array}{ll}
0 & a \\
0 & 0
\end{array}\Big) : a \in F \Big\},$$ and since $\rad(I) = 0$ ($I$ is a simple ring), it follows from Theorem~\ref{radE} that $\rad(E(R,I)) = K$. Clearly $\rad(E(R,I)) \neq \rad(R) \oplus \rad(I)$.
\end{example}

Let us now give a sketch of a similar but more interesting example, in which the set of first coordinates of $\rad(E(R,I))$ is a proper subset of $\rad(R)$ and $I$ does not have a unit.

\begin{example} \label{matrixeg2}
Let $F$ be a field, let $V$ be a countably infinite-dimensional $F$-vector space, and let $\End_F (V)$ denote the endomorphism ring of $V$. We identify $\End_F (V)$ with the ring of row-finite matrices over $F$. Now, let $R \subseteq \End_F (V)$ be the subring of all upper-triangular matrices, and let $I \subseteq \End_F (V)$ be the subrng of all matrices having only finitely many nonzero entries. It is easy to verify that $I$ is an $R$-rng and that $\rad(R)$ is the subset of all strictly upper-triangular matrices (i.e., ones with zeros everywhere on the main diagonal). Also, $I$ has no ideals other than $0$ and $I$, so it can be quickly seen that $\rad(I) = 0$ (e.g., the matrix which has $-1$ in upper left corner and zeros elsewhere cannot be left quasi-regular). Hence, by Theorem~\ref{radE}, $(r, i) \in \rad(E(R,I))$ if and only if $r \in \rad(R)$ and $rj = -ij$ for all $j \in I$, which can only happen if $r = -i$ (as elements of $\End_F (V)$). Therefore $\rad(E(R,I)) = \{(r, -r) : r \in \rad(R) \cap I\}$ (i.e., the set of $(r, -r)$, where $r$ is strictly upper-triangular and has only finitely many nonzero entries). 
\end{example}

\section{The upper nil radical} \label{nilradsection}

We recall that the {\em upper nil radical} of a rng $I$, denoted $\Nil(I)$, is the sum of all the nil ideals of $I$. This is an ideal of $I$ which can also be described as $\{i \in I : \langle i \rangle \text{ is nil} \}$ (where $\langle i \rangle$ denotes the ideal generated by $i$). The goal of this section is to describe $\Nil(E(R,I))$ for arbitrary $R$ and $I$. At first glance, it might seem natural to work with nil $R$-ideals of $I$ for this purpose, rather than nil ideals of $I$, since the former are ideals of $E(R,I)$, while the latter are not. However, the next lemma shows that nil $R$-ideals and nil ideals are, in some sense, interchangeable.

\begin{lemma} \label{R-ideal-ideal}
Let $R$ be a ring, let $I$ be an $R$-rng, and let $J$ be an ideal of $I$. Also, let $K$ be the $R$-ideal of $I$ generated by $J$. Then $K = RJR$ and $K^3 \subseteq J$.

In particular, $J$ is nil $($respectively, nilpotent\,$)$ if and only if $K$ is nil $($respectively, nilpotent\,$)$. Also, $\,\Nil(I)$ is an $R$-ideal of $I$.
\end{lemma}

\begin{proof}
First, we note that $I(RJR) = (IR)JR = IJR \subseteq JR \subseteq RJR$, and similarly $(RJR)I \subseteq RJR$. This implies that $RJR$ is an ideal of $I$, and hence an $R$-ideal. Since $RJR \subseteq K$, the two must in fact be equal. Also, $K^3 \subseteq IKI = I(RJR)I \subseteq J$. The remaining claims are now immediate, since the $R$-ideal of $I$ generated by $\Nil(I)$ must be nil and hence contained in $\Nil(I)$.
\end{proof}

In order to describe $\Nil(E(R,I))$, we need one more easy lemma.

\begin{lemma}
Let $R$ be a ring, let $I$ be an $R$-rng, let $(r, i) \in E(R,I)$ be any element, and let $n$ be a positive integer. Then $(r^n, jr + ji + r^{n-1}i) = (r,i)^n = (r^n, rk + ik + ir^{n-1})$ for some $j, k \in I$.
\end{lemma}

\begin{proof}
We will only prove the first equality, since the second follows by symmetry. We proceed by induction on $n$. The claim clearly holds for $n = 1$ (with $j = 0$). Assuming that $(r,i)^n = (r^n, jr + ji + r^{n-1}i)$ for some $n \geq 1$, we have $$(r,i)^{n+1} = (r,i)^n(r,i) = (r^{n+1}, (jr + ji + r^{n-1}i)r + (jr + ji + r^{n-1}i)i + r^ni),$$ as desired.
\end{proof}

\begin{theorem} \label{nilE} 
Let $R$ be a ring, let $I$ be an $R$-rng, and let $r \in R$, $i \in I$ be any elements. Then the following are equivalent:
\begin{enumerate}
\item[$(1)$] $(r, i) \in \Nil(E(R,I))$,
\item[$(2)$] $r \in \Nil(R)$ and $jr+ji \in \Nil(I)$ for all $j \in I$,
\item[$(3)$] $r \in \Nil(R)$ and $rj+ij \in \Nil(I)$ for all $j \in I$.
\end{enumerate}
Moreover, $(r, 0) \in \Nil(E(R,I))$ if and only if $r \in \{r \in \Nil(R) : rI + Ir \subseteq \Nil(I)\}$, and $(0, i) \in \Nil(E(R,I))$ if and only if $i \in \Nil(I)$.
\end{theorem}

\begin{proof}
Suppose that $(r, i) \in \Nil(E(R,I))$. Then the set of first coordinates of $\langle (r, i) \rangle$ is the ideal of $R$ generated by $r$. Since $\langle (r, i) \rangle$ is nil, this implies that $r \in \Nil(R)$. Now, let $j \in I$ be any element. Then $(0,j)(r,i) = (0,jr+ji)$, and hence $\langle (0,jr+ji) \rangle \subseteq \langle (r,i) \rangle$. This shows that the $R$-ideal of $I$ generated by $jr+ji$ is nil, and therefore that the ideal of $I$ generated by  $jr+ji$ is nil. That is, $jr+ji \in \Nil(I)$. Similarly, $(r,i)(0,j) = (0,rj+ij)$ implies that $rj+ij \in \Nil(I)$. Thus (1) implies (2) and (3).

Now, let us set $$N = \{(r,i) : r \in \Nil(R) \text{ and } jr+ji \in \Nil(I) \text{ for all } j \in I \}.$$ It is easy to see that $N$ is an ideal of $E(R,I)$. (Let $k \in I$ be an arbitrary element. Then $(r,i)(p,j) = (rp, ip + rj + ij)$ and $$k(rp) + k(ip + rj + ij) = (kr + ki)p + (kr + ki)j.$$ Using the fact that $\Nil(R) \subseteq R$ is an ideal and $\Nil(I) \subseteq I$ is an $R$-ideal (by Lemma~\ref{R-ideal-ideal}), these equalities show that $N$ is closed under multiplication by elements of $E(R,I)$ both on the left and on the right (first assuming that $(p,j) \in N$ and then that $(r,i) \in N$). Also, if $(r,i), (p,j) \in N$, then $(r,i) + (p,j) = (r+p, i+j)$, and for all $k \in I$, we have $k(r+p) + k(i+j) = (kr + ki) + (kp + kj) \in \Nil(I)$. This shows that $N$ is closed under addition as well.) Next, we note that every element of $N$ is nilpotent. For, let $(r,i) \in N$, and let $n$ be a positive integer such that $r^n = 0$. Then $(r,i)^{n+1} = (0,jr + ji)$ for some $j \in I$, by the previous lemma. By hypothesis, this is an element of $\Nil(I)$ and hence nilpotent. It follows that $N \subseteq \Nil(E(R,I))$, which shows that (2) implies (1).

A similar argument, with $$\{(r,i) : r \in \Nil(R) \text{ and } rj+ij \in \Nil(I) \text{ for all } j \in I \}$$ in place of $N$, shows that (3) implies (1). The final claim follows from the equivalence of (1), (2), and (3), and also from the fact that $\Nil(I)$ is an $R$-ideal (so the $R$-ideal generated by an element of $\Nil(I)$ must be nil).
\end{proof}

One can, of course, deduce the obvious analogue of Corollary~\ref{radcor} for upper nil radicals from this result. Also, in Examples~\ref{matrixeg} and~\ref{matrixeg2} one can find rings $E(R,I)$ such that $\Nil(E(R,I)) \neq \Nil(R) \oplus \Nil(I)$, since in those examples $\rad(R) = \Nil(R)$, $\rad(I) = \Nil(I)$, and $\rad(E(R,I)) = \Nil(E(R,I))$.

\section{Semiprimeness} \label{semiprimesection}

In this section we will determine when a ring of the form $E(R,I)$ is semiprime and then discuss some special cases. We begin with a quick but useful lemma.

\begin{lemma} \label{multinjective}
Let $R$ be a ring, and let $I$ be a semiprime $R$-rng. Also, let $J$ be a nonzero $R$-subrng of $I$, and let $\varphi : J \rightarrow R$ be an $R$-homomorphism such that for all $i \in I$, $j \in J$ one has $ij = i\varphi(j)$, $ji = \varphi(j)i$. Then $\varphi$ must be injective.
\end{lemma}

\begin{proof}
Suppose that $\varphi(j) = 0$ for some $j \in J$. Then $jI = Ij = 0$, and hence $j \in \ann_I(I) = \{i \in I : iI = Ii = 0\}$, which is an ideal of $I$ that has square zero. 
Since $I$ is semiprime, $ \ann_I(I) = 0$, showing that $\varphi$ is injective.
\end{proof}

The next result is a generalization of \cite[Proposition 3.2]{DM} to arbitrary $R$-rngs.

\begin{theorem}  \label{semiprime}  Let $R$ be a ring, and let $I$ be an $R$-rng. Then $E(R,I)$ is semiprime if and only if the following three conditions hold:
\begin{enumerate}
\item[$(1)$] $I$ is a semiprime rng,
\item[$(2)$] there are no nonzero ideals $A$ of $R$ such that $A^2 = 0$ and $A \subseteq \ann_R(I)$, and 
\item[$(3)$] there do not exist a nonzero $R$-subrng $J \subseteq I$ and an {\em injective} $R$-homomorphism $\varphi : J \rightarrow R$ such that $J^2 = 0$, and for all $i \in I$, $j \in J$ one has $ij = i\varphi(j)$, $ji = \varphi(j)i$.
\end{enumerate}
Moreover, the above statement also holds if $\, (3)$ is replaced with
\begin{enumerate}
\item[$(3')$] there do not exist a nonzero $R$-subrng $J \subseteq I$ and an $R$-homomorphism $\varphi : J \rightarrow R$ such that $J^2 = 0$, and for all $i \in I$, $j \in J$ one has $ij = i\varphi(j)$, $ji = \varphi(j)i$.
\end{enumerate}
In particular, if $R$ and $I$ are semiprime, then $E(R,I)$ is semiprime. 
\end{theorem}

\begin{proof} 
Suppose that $I$ is not semiprime. Then we can find some nonzero ideal $J \subseteq I$ such that $J^2 = 0$. Let $K$ be the (nonzero) $R$-ideal generated by $J$. By Lemma~\ref{dirsumideal}, $0\oplus K$ is an ideal of $E(R,I)$, and, by Lemma~\ref{R-ideal-ideal}, it is nilpotent. This implies that $E(R,I)$ is not semiprime.

Next, suppose that there is a nonzero ideal $A$ of $R$ such that $A^2 = 0$ and $A \subseteq \ann_R(I)$. By Lemma~\ref{dirsumideal}, $A \oplus 0$ is an ideal of $E(R,I)$, and clearly $(A \oplus 0)^2 = 0$. Therefore, $E(R, I)$ is again not semiprime.

Now, suppose that there is a nonzero $R$-subrng $J \subseteq I$ and an injective $R$-homomorphism $\varphi : J \rightarrow R$ such that $J^2 = 0$, and for all $i \in I$, $j \in J$ one has $ij = i\varphi(j)$, $ji = \varphi(j)i$. Let $K = \{(a, -j) : j \in J, a = \varphi(j)\}$. Then $K$ is an ideal of $E(R,I)$, by Proposition~\ref{idealdescription} (with $A = \varphi(J)$ and $Z = 0$). Since $J \neq 0$, $K$ is also nonzero. On the other hand, since $J^2 = 0$, we have $K^2 = 0$, and hence $E(R,I)$ is not semiprime once more.

We have shown that if $E(R,I)$ is semiprime, then the conditions (1), (2), and (3) must hold. For the converse, let us assume (1) and (2), and show that if $E(R,I)$ is not semiprime, then (3) must be false.  Thus, let $K$ be a nonzero ideal of $E$ such that $K^2 = 0$. By Proposition~\ref{idealdescription}, we can write $K = \{(a, -j): a \in A, j \in J, a+Z = \varphi(j)\}$, where $Z \subseteq A$ are ideals of $R$,  $J$ is an $R$-subrng of $I$, and $\varphi: J \rightarrow A/Z$ is a surjective $R$-homomorphism satisfying conditions (a) and (b) of the proposition. Now, $0 \oplus \ker(\varphi) \subseteq K$, and hence $\ker(\varphi)^2 = 0$. But, by the Proposition~\ref{idealdescription}, $\ker(\varphi)$ is an $R$-ideal of $I$, and hence $\ker(\varphi) = 0$, since we have assumed that $I$ is semiprime.  Thus, $\varphi$ is an injective $R$-homomorphism.

Now, since $\varphi$ is injective, by Proposition~\ref{idealdescription}, $Z \subseteq \ann_R(I) \cap A$. But, $\ann_R(I) \cap A = 0$, since $A^2 = 0$ and we've assumed that (2) holds. Thus, $Z = 0$, and so $\varphi$ is an (injective) $R$-homomorphism $J \rightarrow R$.

Let $(a, -j) \in K$ be any element. Then, by the previous paragraph, we can write $a = \varphi(j)$. By Proposition~\ref{idealdescription}, for all $i\in I$ we have $\varphi(j)i - ji, i\varphi(j) - ij \in \ker(\varphi) = 0$. Hence $ij = i\varphi(j)$ and $ji = \varphi(j)i$.

Next, let $i, j \in J$ be any two elements. Then $(\varphi(i), -i), (\varphi(j), -j) \in K$, and since $K^2 = 0$, we have $$0 = (\varphi(i), -i)(\varphi(j), -j) = (\varphi(ij), -\varphi(i)j - i\varphi(j) + ij).$$ Since $\varphi$ is injective, $0 = \varphi(ij)$ implies that $ij = 0$. From this we conclude that $J^2 = 0$.

Finally, we note that $J \neq 0$, since otherwise we would have $A = 0$ and hence also $K = 0$. Thus, assuming that $E(R,I)$ is not semiprime, we have constructed $J$ and $\varphi$ that violate (3), concluding the proof of the main claim.

It is clear that $(3')$ implies (3). Also, by Lemma~\ref{multinjective}, (1) and (3) imply $(3')$. Hence, in our situation, (3) and $(3')$ are interchangeable.

For the final statement, we note that if there were a nonzero $R$-subrng $J \subseteq I$ and an injective $R$-homomorphism $\varphi : J \rightarrow R$ such that $J^2 = 0$, then $R$ would possess a nonzero ideal with square zero (namely $\varphi(J)$). Thus, if $R$ is semiprime, then (2) and (3) must be satisfied.
\end{proof}

The ideal $K$ in Example~\ref{matrixeg} is an instance of a nonzero ideal in a ring of the form $E(R,I)$, such that $K^2 = 0$, but where $K$ is neither of the form $A \oplus 0$ for some ideal $A$ of $R$ nor of the form $0 \oplus J$ for some $R$-ideal of $I$. Condition (3) in the above statement addresses this sort of obstacle to being semiprime.

From the previous result we can quickly derive a criterion for ideals of $E(R,I)$ having a certain form to be semiprime.

\begin{corollary} \label{semiprimeA+J}
Let $R$ be a ring, and let $I$ be an $R$-rng. Also, let $J$ be an $R$-ideal of $I$, and $A$ an ideal of $R$ such that $AI + IA \subseteq J$ $($so that $A \oplus J$ is an ideal of $E(R,I)$, by Lemma~\ref{dirsumideal}\,$)$. Then $K = A \oplus J$ is a semiprime ideal if and only if the following three conditions hold:
\begin{enumerate}
\item[$(1)$] for all ideals $L \subseteq I$ that contain $J$, if $L^2 \subseteq J$, then $L = J$,
\item[$(2)$] there does not exist an ideal $B \subseteq R$, that properly contains $A$, such that $B^2 \subseteq A$ and $BI + IB \subseteq J$, and
\item[$(3)$] there do not exist an $R$-subrng $L \subseteq I$, that properly contains $J$ and satisfies $L^2 \subseteq J$, and an $R$-homomorphism $\varphi : L/J \rightarrow R/A$ such that for all $i \in I/J$, $j \in L/J$ one has $ij = i\varphi(j)$, $ji = \varphi(j)i$.
\end{enumerate}
\end{corollary}

\begin{proof}
We note that $A \oplus J$ is a semiprime ideal of $E(R,I)$ if and only if $E(R,I)/(A \oplus J) \cong E(R/A, I/J)$ is a semiprime ring. Hence, the claim follows from an application of Theorem~\ref{semiprime} to $E(R/A, I/J)$, once we note that since $A \subseteq \ann_R(I/J)$, in the above situation the concepts of ``$R/A$-subrng of $I/J$" and ``$R/A$-homomorphism" coincide with those of ``$R$-subrng of $I/J$" and ``$R$-homomorphism," respectively.
\end{proof}

The rest of this section is devoted to the relationship between $R$ being semiprime and $E(R,I)$ being semiprime. The next lemma gives a partial converse to the last statement of Theorem~\ref{semiprime}, in view of the fact that if $E(R, I)$ is semiprime, then so is $I$.

\begin{lemma} \label{annsemiprime}
Let $R$ be a ring, and let $I$ be an $R$-rng such that $\, \ann_R(I)$ is a semiprime ideal of $R$. If $E(R,I)$ is semiprime, then so is $R$.
\end{lemma}

\begin{proof}
Suppose that $E(R,I)$ is semiprime, and let $A$ be a nilpotent ideal of $R$.  Then $A \subseteq \ann_R(I)$, since the image of $A$ in the semiprime ring $R/\ann_R(I)$ is nilpotent and hence zero.  Also, $A \oplus 0$ is
an ideal of $E(R,I)$, by Lemma~\ref{dirsumideal}, and it is nilpotent. We conclude that $A = 0$, showing that $R$ is semiprime.  
\end{proof}

\begin{lemma}\label{lemmacentralprime}
Let $R$ be a ring, and let $I$ be a centrally generated $R$-rng. If $I$ is a prime $($respectively, semiprime\,$)$ rng, then $\, \ann_R(I)$ is a prime $($respectively, semiprime\,$)$ ideal of $R$.
\end{lemma}

\begin{proof}
We will prove only the ``prime" version of the statement, since the ``semiprime" version can be proved similarly.  Thus, suppose that $I$ is prime, and let $A$ and $B$ be ideals of $R$ such that $AB \subseteq \ann_R(I)$. By Lemma~\ref{centrallemma}, $AI = IA$ and $BI = IB$ are $R$-ideals of $I$. Then, $(AI)(BI) = (IA)(BI) = I(AB)I = 0$. Therefore, either $AI = 0$ or $BI = 0$, showing that either $A \subseteq \ann_R(I)$ or $B \subseteq \ann_R(I)$. Thus $\ann_R(I)$ is a prime ideal of $R$.
\end{proof}

Putting the previous two lemmas together, we obtain another partial converse to the last statement of Theorem~\ref{semiprime}. (This claim can also be deduced, more directly, from \cite[Lemma 4.8]{DM}.)

\begin{corollary}\label{centralsemiprime}
Let $R$ be a ring, and let $I$ be a centrally generated $R$-rng. If $E(R,I)$ is semiprime, then so is $R$.
\end{corollary}

In general, $E(R,I)$ can be semiprime without $R$ being such, as the next example shows.

\begin{example}
Let $S$ be any nonzero ring, let $R$ be the subring of $S \langle x, y \rangle /(x^2)$ generated by $S$ and $x$ (so $R \cong S[x]/(x^2)$), and let $I$ be the ideal of $S \langle x, y \rangle /(x^2)$ generated by $y$. It is easy to see that $E(R,I) \cong S \langle x, y \rangle /(x^2)$, and hence that $E(R,I)$ is (semi)prime. (For all nonzero $f, g \in S \langle x, y \rangle /(x^2)$ one has $fyg \neq 0$.) On the other hand, $R$ clearly is not semiprime.
\end{example}

\section{Primeness}  \label{primesection}

Let us now characterize when $E(R,I)$ is prime. The proof of the next result, which is a generalization of \cite[Proposition 3.3]{DM} to arbitrary $R$-rngs, is very similar to that of Theorem~\ref{semiprime}.

\begin{theorem}  \label{prime}  Let $R$ be a ring, and let $I$ be a nonzero $R$-rng. Then $E(R,I)$ is prime if and only if the following three conditions hold:
\begin{enumerate}
\item[$(1)$] $I$ is a prime rng,
\item[$(2)$] $\ann_R(I) = 0$, and 
\item[$(3)$] there do not exist a nonzero $R$-subrng $J \subseteq I$ and an {\em injective} $R$-homomorphism $\varphi : J \rightarrow R$ such that for all $i \in I$, $j \in J$ one has $ij = i\varphi(j)$, $ji = \varphi(j)i$.
\end{enumerate}
Moreover, the above statement also holds if $\, (3)$ is replaced with
\begin{enumerate}
\item[$(3')$] there do not exist a nonzero $R$-subrng $J \subseteq I$ and an $R$-homomorphism $\varphi : J \rightarrow R$ such that for all $i \in I$, $j \in J$ one has $ij = i\varphi(j)$, $ji = \varphi(j)i$.
\end{enumerate}
\end{theorem}

\begin{proof}
Suppose that $I$ is not prime. Then we can find nonzero ideals $J_1, J_2 \subseteq I$ such that $J_1J_2 = 0$. Let $K_1$ and $K_2$ be the (nonzero) $R$-ideals of $I$ generated by $J_1$ and $J_2$, respectively. By Lemma~\ref{dirsumideal}, $0\oplus K_1$ and $0 \oplus K_2$ are ideals of $E(R,I)$, and, by Lemma~\ref{R-ideal-ideal}, $(0\oplus K_1)^3(0 \oplus K_2)^3 = 0$. This implies that $E(R,I)$ is not prime.

Next, suppose that $\ann_R(I) \neq 0$. By Lemma~\ref{dirsumideal}, $\ann_R(I) \oplus 0$ and $0 \oplus I$ are ideals of $E(R,I)$, and clearly $(\ann_R(I) \oplus 0)(0 \oplus I) = 0$. Therefore, $E(R, I)$ is again not prime.

Now, suppose that there are a nonzero $R$-subrng $J \subseteq I$ and an injective $R$-homomorphism $\varphi : J \rightarrow R$ such that for all $i \in I$, $j \in J$ one has $ij = i\varphi(j)$, $ji = \varphi(j)i$. Let $L = \{(a, -j) : j \in J, a = \varphi(j)\}$. Then $L$ is an ideal of $E(R,I)$, by Proposition~\ref{idealdescription} (with $Z = 0$ and $A = \varphi(J)$). Since $J \neq 0$, $L$ is also nonzero. Now, let $(\varphi(j), -j) \in L$ and $i \in I$ be any two elements. Then $$(\varphi(j), -j) (0, i) =  (0, \varphi(j)i - ji) = 0,$$ by our assumptions on $\varphi$. Hence, we have $L(0\oplus I) = 0$, and therefore $E(R,I)$ is not prime once more.

We have shown that if $E(R,I)$ is prime, then the conditions (1), (2), and (3) must hold. For the converse, let us assume (1) and (2), and show that if $E(R,I)$ is not prime, then (3) must be false.  Thus, let $L$ and $M$ be nonzero ideals of $E(R,I)$ such that $LM = 0$. By Proposition~\ref{idealdescription}, we can write $L = \{(a, -j): a \in A, j \in J, a+Y = \varphi(j)\}$ and $M = \{(b, -k): a \in B, k \in K, b+Z = \psi(k)\}$, where $Y \subseteq A$ and $Z \subseteq B$ are ideals of $R$,  $J$ and $K$ are $R$-subrngs of $I$, and $\varphi: J \rightarrow A/Y$, $\psi: K \rightarrow B/Z$ are surjective $R$-homomorphisms satisfying conditions (a) and (b) of the proposition. 

Now, $0 \oplus \ker(\varphi) \subseteq L$ and $0 \oplus \ker(\psi) \subseteq M$, which implies that $\ker(\varphi)\ker(\psi) = 0$. But, by the Proposition~\ref{idealdescription}, $\ker(\varphi)$ and $\ker(\psi)$ are $R$-ideals of $I$, and hence one of them must be zero, since we have assumed that $I$ is prime.  Thus, let us suppose that $\ker(\varphi) = 0$ (the proof in the other case is similar), and therefore that $\varphi$ is an injective $R$-homomorphism.

Since $\varphi$ is injective, Proposition~\ref{idealdescription} implies that $Y \subseteq \ann_R(I)$. Thus $Y = 0$, since we have assumed that $\ann_R(I) = 0$, and hence $\varphi$ is an (injective) $R$-homomorphism $J \rightarrow R$.

Let $(a, -j) \in L$ be any element. Then, by the previous paragraph, we can write $a = \varphi(j)$. By Proposition~\ref{idealdescription}, for all $i\in I$ we have $\varphi(j)i - ji, i\varphi(j) - ij \in \ker(\varphi) = 0$. Hence $ij = i\varphi(j)$ and $ji = \varphi(j)i$.

Finally, we note that $J \neq 0$, since otherwise, we would have $A = 0$, and hence also $L = 0$. Thus, assuming that $E(R,I)$ is not semiprime, we have constructed $J$ and $\varphi$ that violate (3), concluding the proof of the main claim.

For the last statement, we note that $(3')$ clearly implies (3), and by Lemma~\ref{multinjective}, (1) and (3) imply $(3')$. Hence, in our situation, (3) and $(3')$ are interchangeable.
\end{proof}

The above theorem implies, for instance, that for any ring $R$ and any ideal $I \subseteq R$, $E(R, I)$ is not prime. Thus, unlike the case of semiprime rings, $E(R, I)$ need not be prime when $R$ and $I$ are prime, even when $\ann_R(I) = 0$. However, as in the semiprime case, if $I$ is a centrally generated $R$-rng, then $R$ is prime whenever $E(R,I)$ is. This follows from Theorem~\ref{prime} and Lemma~\ref{lemmacentralprime}, and is also shown directly in \cite[Lemma 4.8]{DM}.

We finish the section with an analogue of Corollary~\ref{semiprimeA+J} for prime ideals.

\begin{corollary}\label{A+J}
Let $R$ be a ring, and let $I$ be a nonzero $R$-rng. Also, let $J$ be a proper $R$-ideal of $I$, and $A$ an ideal of $R$ such that $AI + IA \subseteq J$ $($so that $A \oplus J$ is an ideal of $E(R,I)$, by Lemma~\ref{dirsumideal}\,$)$. Then $K = A \oplus J$ is a prime ideal if and only if the following three conditions hold:
\begin{enumerate}
\item[$(1)$] for all ideals $L_1, L_2 \subseteq I$ that contain $J$, if $L_1L_2 \subseteq J$, then either $L_1 = J$ or $L_2 = J$,
\item[$(2)$] $A = \{r \in R : rI + Ir \subseteq J\}$, and 
\item[$(3)$] there do not exist an $R$-subrng $L \subseteq I$, properly containing $J$, and an $R$-homomorphism $\varphi : L/J \rightarrow R/A$ such that for all $i \in I/J$, $j \in L/J$ one has $ij = i\varphi(j)$, $ji = \varphi(j)i$.
\end{enumerate}
\end{corollary}

\begin{proof}
We note that $A \oplus J$ is a prime ideal of $E(R,I)$ if and only if $E(R,I)/(A \oplus J) \cong E(R/A, I/J)$ is a prime ring. Hence, the result follows from an application of Theorem~\ref{prime} to $E(R/A, I/J)$, once we note that $\ann_{R/A}(I/J) = 0$ if and only if $A = \{r \in R : rI + Ir \subseteq J\}$, and that since $A \subseteq \ann_R(I/J)$, in the above situation the concepts of ``$R/A$-subrng of $I/J$" and ``$R/A$-homomorphism" coincide with those of ``$R$-subrng of $I/J$" and ``$R$-homomorphism," respectively.
\end{proof}

\section{Prime ideals in a special case} \label{primeidealsection}

Let us now turn to the questions of when an arbitrary ideal of $E(R,I)$ is prime and when it is maximal. We will answer these questions in the case where the $R$-rng $I$ comes equipped with an $R$-homomorphism $\varphi : I \rightarrow R$ that satisfies $i\varphi(j) = ij = \varphi(i)j$ for all $i, j \in I$. This condition may look esoteric at first glance, but it is satisfied, for instance, by all injective $R$-homomorphisms $\varphi : I \rightarrow R$. For, $i\varphi(j)$, $ij$, and $\varphi(i)j$ all have the same image under an $R$-homomorphism $\varphi : I \rightarrow R$. Hence, if $\varphi$ is injective, then the three must be equal. 

If $I$ is semiprime, then an $R$-homomorphism $\varphi : I \rightarrow R$ that satisfies the above condition must be injective (cf.\ Lemma~\ref{multinjective}). However, in general, an $R$-homomorphism $\varphi : I \rightarrow R$ may satisfy this condition without being injective, as the next example shows.

\begin{example}
Let $S$ be any nonzero ring, let $R = S [x]$, and let $I$ be the ideal of $S [x,y]/(xy, y^2)$ generated by $x$ and $y$. Then we can define an $R$-homorphism $\varphi : I \rightarrow R$ via $\varphi (f(x) + g(y)) = f(x)$, where $f(x) = a_1x+ \dots + a_nx^n$ and $g(y) = by$, for some $a_1, \dots, a_n, b \in S$. Now, let $f_1(x) + g_1(y)$ and $f_2(x) + g_2(y)$ be arbitrary elements of $I$ (where $f_1(x), g_1(y), f_2(x), g_2(y)$ have appropriate forms, as above). Then, $$(f_1(x) + g_1(y))(f_2(x) + g_2(y)) = f_1(x)f_2(x)$$ $$= (f_1(x) + g_1(y))f_2(x) = (f_1(x) + g_1(y))\varphi (f_2(x) + g_2(y)).$$ Hence for all $i, j \in I$ we have $i\varphi(j) = ij$, and similarly $\varphi(i)j = ij$. But, $\varphi$ clearly is not injective.
\end{example}

We now proceed to give a classification of the prime ideals of $E(R,I)$ in the case where there is an $R$-homomorphism $\varphi : I \rightarrow R$ that satisfies $i\varphi(j) = ij = \varphi(i)j$ for all $i, j \in I$. The argument requires several short lemmas.

\begin{lemma}\label{imposter}
Let $R$ be a ring, let $I$ be an $R$-rng, and let $\varphi : I \rightarrow R$ be an $R$-homomorphism such that for all $i, j \in I$ one has $i\varphi(j) = ij = \varphi(i)j$. Also, set $I_\varphi = \{(\varphi(i), -i) : i \in I \} \subseteq E(R,I)$. Then there exists an $R$-isomorphism $\psi : E(R,I) \rightarrow E(R,I)$ such that $\psi (I) = I_\varphi$.
\end{lemma}

\begin{proof}
Let us define a map $\psi : E(R,I) \rightarrow E(R,I)$ by $\psi (r,i) = (r + \varphi(i), -i)$, for all $r \in R$ and $i \in I$. Since $\psi^2$ is the identity map, it follows that $\psi$ is a bijection. Also, it is clear that $\psi (I) = I_\varphi$ and $\psi(R) = R$. Thus, to finish the proof it suffices to check that $\psi$ is a ring homomorphism, and this verification is routine. 
\end{proof}

\begin{lemma}\label{diag}
Let $R$ be a ring, let $I$ be an $R$-rng, and let $\varphi : I \rightarrow R$ be an $R$-homomorphism such that for all $i, j \in I$ one has $i\varphi(j) = ij = \varphi(i)j$. Also, let $J$ be a nonzero $R$-ideal of $I$, and set $J_\varphi = \{(\varphi(j), -j) : j \in J \} \subseteq E(R,I)$. Then the ideal $J_\varphi$ is prime if and only if $0 \oplus J$ is prime if and only if $J = I$ and $R$ is a prime ring.
\end{lemma}

\begin{proof}
By, Lemma~\ref{imposter}, there exists an $R$-isomorphism $\psi : E(R,I) \rightarrow E(R,I)$ such that $\psi (I) = I_\varphi$. In particular, $\psi (0\oplus J) = J_\varphi$, which implies that $J_\varphi$ is prime if and only if $0 \oplus J$ is prime.

If $J = I$, then $E(R,I)/(0 \oplus J) \cong R$ is a prime ring if and only if $0 \oplus J \subseteq E(R,I)$ is a prime ideal. Thus, to finish the proof, it suffices to show that if $0 \oplus J$ is prime, then $J = I$. 

Suppose that $0 \oplus J$ is prime and $J \neq I$. Then, by Corollary~\ref{A+J}, $\{r \in R : rI + Ir \subseteq J\} = 0$. Now for all $i \in I$ and $j \in J$ we have $i\varphi(j) = ij \in J$ and $\varphi(j)i = ji \in J$, and hence $\varphi(J) = 0$. Thus, $\varphi$ restricts to an $R$-homomorphism $\bar{\varphi} : I/J \rightarrow R$ such that for all $i, j \in I/J$ one has $ij = i\bar{\varphi}(j)$, $ji = \bar{\varphi}(j)i$, which contradicts Corollary~\ref{A+J}. Therefore, $J = I$.
\end{proof}

\begin{lemma}\label{diag2}
Let $R$ be a ring, let $I$ be an $R$-rng, and let $\varphi : I \rightarrow R$ be an $R$-homomorphism such that for all $i, j \in I$ one has $i\varphi(j) = ij = \varphi(i)j$. Also, let $J$ be a nonzero $R$-ideal of $I$, let $\psi : J \rightarrow R$ be an $R$-homomorphism such that for all $i \in I$ and $j \in J$ one has $i\psi(j) = ij$ and $\psi(j)i = ji$, and set $J_\psi = \{(\psi(j), -j) : j \in J \} \subseteq E(R,I)$. Then the ideal $J_\psi$ is prime if and only if $J = I$, $\varphi (i) = \psi (i)$ for all $i \in I$, and $R$ is a prime ring.
\end{lemma}

\begin{proof}
It suffices to prove that if $J_\psi$ is prime, then $\psi$ is the restriction of $\varphi$ to $J$, since the desired result will then follow from Lemma~\ref{diag}.

Suppose that $J_\psi$ is prime, and let $j \in J$ be any element. Then for all $i \in I$, $\psi(j)i = ji = \varphi(j)i$, and hence $(\psi(j)-\varphi(j))I = 0$ (and similarly $I(\psi(j)-\varphi(j)) = 0$). Therefore $\psi(j)-\varphi(j) \in \ann_R(I)$. Now, since $(\ann_R(I) \oplus 0)I = 0$, and since $J_\psi$ is prime, either $I \subseteq J_\psi$ or $\ann_R(I) \oplus 0 \subseteq J_\psi$. In the first case, $J = I$, and $\psi(i) = 0$ for all $i \in I$, implying that $I^2 = 0$. Since $E(R,I)/J_\psi = E(R,I)/I \cong R$ is prime, this implies that $\varphi(I) = 0$, and in particular, $\varphi(i) = \psi(i)$ for all $i \in I$. Let us therefore assume that $\ann_R(I) \oplus 0 \subseteq J_\psi$. Then for all $j \in J$ we have $(\psi(j)-\varphi(j), 0) \in J_\psi$, implying that $\psi(j) = \varphi(j)$. Thus, in every case, $\psi$ is the restriction of $\varphi$ to $J$, as desired.
\end{proof}

We note that if $I$ is an $R$-rng and $\varphi : I \rightarrow R$ is an $R$-homomorphism such that for all $i, j \in I$ one has $i\varphi(j) = ij = \varphi(i)j$, then $J \subseteq I$ is an $R$-subrng of $I$ if and only if it is an $R$-ideal of $I$. This is because an $R$-ideal is always and $R$-subrng, and conversely, if $J$ is an $R$-subrng of $I$, then for all $i \in I$ and $j \in J$ we have $ji = j \varphi(i) \in J$, $ij = \varphi(i)j \in J$, showing that $IJ + JI \subseteq J$.

\begin{theorem}\label{primeideal}
Let $R$ be a ring, let $I$ be an $R$-rng, and let $\varphi : I \rightarrow R$ be an $R$-homomorphism such that for all $i, j \in I$ one has $i\varphi(j) = ij = \varphi(i)j$. An ideal $K \subseteq E(R, I)$ is prime if and only if it has one of the following two forms.
\begin{enumerate}
\item[$(1)$] $A \oplus I$ for some prime ideal $A$ of $R$.
\item[$(2)$] $\{(a, -i) : i \in I, a \in R, \text{ such that } a - \varphi(i) \in Z\}$, where $Z$ is a prime ideal of $R$ and $\varphi (I) \not\subseteq Z$.
\end{enumerate}
\end{theorem}

\begin{proof}
As before, an ideal of the form (1) must be prime, since $E(R,I)/(A \oplus I) \cong R/A$. Now, suppose that $K \subseteq E(R, I)$ is an ideal of the form (2), that is $K = \{(a, -i) : i \in I, a \in R, a + Z = \pi \varphi(i) \}$, where $\pi : R \rightarrow R/Z$ is the natural projection. (We note that since $ZI + IZ \subseteq \{i \in I : \varphi(i) \in Z\} = \ker(\pi \varphi)$, for all $(a, -i) \in K$ and $j \in I$, $aj -ij = (a - \varphi(i))j$ and $ja - ji = j(a - \varphi(i))$ are elements of $\ker(\pi \varphi)$.) By Proposition~\ref{idealdescription}, $Z \oplus \ker(\pi \varphi)$ is a subideal of $K$. Hence $K$ is prime if and only if its image $\bar{K}$ in $E(R, I) / (Z \oplus \ker(\pi \varphi)) \cong E(R/Z, I/\ker(\pi \varphi))$ is prime. Now, the map $\pi \varphi$ restricts to an $R/Z$-homomorphism $\bar{\varphi} : I/\ker(\pi \varphi) \rightarrow R/Z$ that satisfies $i\bar{\varphi}(j) = ij = \bar{\varphi}(i)j$ for all $i,j \in I/\ker(\pi \varphi)$ (since $\bar{\varphi}$ is injective; cf.\ the remark at the beginning of the section). Noting that $R/Z$ is a prime ring and that $\bar{K} = \{(\bar{\varphi}(i),-i) : i \in I/\ker(\pi \varphi)\}$, and applying Lemma~\ref{diag2} to $E(R/Z, I/ \ker(\pi \varphi))$ we conclude that $\bar{K}$, and hence also $K$, must be a prime ideal.

For the converse, suppose that $K$ is a prime ideal of $E(R,I)$, and write $K = \{(a, -j) : a \in A, j \in J, a + Z = \psi(i)\}$ for appropriate $A$, $J$, $Z$, and $\psi$ (as specified in Proposition~\ref{idealdescription}). Let $B = \{r \in E(R,I) : rI + Ir \subseteq K\}$. Then $B$ is an ideal of $E(R,I)$, and since $BI \subseteq K$, we must have either $B \subseteq K$ or $I \subseteq K$. If $I \subseteq K$, then $K = A \oplus I$, and $A$ must be a prime ideal of $R$, since $E/K = E(R,I)/(A \oplus I) \cong R/A$ is a prime ring. Thus, in this case $K$ is of the form (1). Let us therefore assume that $B \subseteq K$, and let $j \in \ker(\psi)$ and $i \in I$ be any elements. Then $\varphi(j)i = ji \in \ker(\psi)$ and $i\varphi(j) = ij \in \ker(\psi)$, since $\ker(\psi)$ is an $R$-ideal of $I$ (by Proposition~\ref{idealdescription}). In particular, this shows that $\varphi(\ker(\psi))I + I\varphi(\ker(\psi))\subseteq K$, and hence $\varphi(\ker(\psi)) \oplus 0 \subseteq B \subseteq K$. More specifically $\varphi(\ker(\psi)) \subseteq Z$. Now, let $L \subseteq I$ be the preimage of $Z$ under $\varphi$, and let $l \in L$, $i \in I$ be any elements. Then $li = \varphi(l)i \in ZI \subseteq \ker(\psi)$ and similarly $il \in \ker(\psi)$. Hence, $0\oplus L \subseteq B \subseteq K$, showing that $L \subseteq \ker(\psi)$. We conclude that $\ker(\psi)$ is the preimage of $Z$ under $\varphi$.

Since $K$ is prime, its image $\bar{K}$ in $E(R, I) / (Z \oplus \ker(\psi)) \cong E(R/Z, I/\ker(\psi))$ is prime. Since $\ker(\psi)$ is the preimage of $Z$ under $\varphi$, and since $ZI + IZ \subseteq \ker(\psi)$, $\varphi$ restricts to an $R/Z$-homomorphism $\bar{\varphi} : I/\ker(\psi) \rightarrow R/Z$ that satisfies $i\bar{\varphi}(j) = ij = \bar{\varphi}(i)j$ for all $i,j \in I/\ker(\psi)$ (since $\bar{\varphi}$ is injective). Also, $\psi$ restricts to an $R/Z$-homomorphism $\bar{\psi} : J/ \ker(\psi) \rightarrow R/Z$ that satisfies $i\bar{\psi}(j) = ij$, $\bar{\psi}(j)i = ji$ for all $i \in I/\ker(\psi)$, $j \in J/\ker(\psi)$. We note that $\bar{K} = \{(\bar{\psi}(j), -j) : j \in J/\ker(\psi)\}$. By Lemma~\ref{diag2}, $R/Z$ is prime, $J/\ker(\psi) = I/\ker(\psi)$, and $\bar{\varphi}(i) = \bar{\psi}(i)$ for all $i \in I/\ker(\psi)$. Thus, $Z$ is a prime ideal of $R$, and since $\ker(\psi) \subseteq J$, we conclude that $J = I$. Further, for all $i \in I$, we have $\psi(i) = \pi \varphi (i)$, where $\pi : R \rightarrow R/Z$ is the natural projection. Therefore, if $\varphi(I) \not\subseteq Z$, then $K$ is of the form (2). Otherwise, $\psi (I) = 0$, and $K$ is of the form (1), by Proposition~\ref{idealdescription}.
\end{proof}

From the above result we can quickly obtain a classification of the maximal ideals of $E(R,I)$ in the situation under consideration.

\begin{corollary} \label{maximalideal}
Let $R$ be a ring, let $I$ be an $R$-rng, and let $\varphi : I \rightarrow R$ be an $R$-homomorphism such that for all $i, j \in I$ one has $i\varphi(j) = ij = \varphi(i)j$. An ideal $K \subseteq E(R, I)$ is maximal if and only if it has one of the following two forms.
\begin{enumerate}
\item[$(1)$] $A \oplus I$ for some maximal ideal $A$ of $R$.
\item[$(2)$] $\{(a, -i) : i \in I, a \in R, \text{ such that } a - \varphi(i) \in Z\}$, where $Z$ is a maximal ideal of $R$ and $\varphi (I) \not\subseteq Z$.
\end{enumerate}
\end{corollary}

\begin{proof}
By Theorem~\ref{primeideal} every maximal ideal of $E(R,I)$ is of the form $\{(a, -i) : i \in I, a \in R, a - \varphi(i) \in A\}$, where $A$ is a prime ideal of $R$ (if $\varphi (I) \subseteq A$, then this is of the form (1) of the theorem). Further, $A$ must be a maximal ideal of $R$, since otherwise $A \subset M$, for some maximal ideal $M$ of $R$, and in this case, $$\{(a, -i) : i \in I, a \in R, a - \varphi(i) \in A\} \subset \{(a, -i) : i \in I, a \in R, a - \varphi(i) \in M\},$$ which is a proper ideal of $E(R,I)$. It follows that every maximal ideal of $E(R,I)$ is of one of the two forms in the statement.

Conversely, an ideal of the form (1) is maximal, since $E(R,I)/(A \oplus I) \cong R/A$. Now, let $K$ be of the form (2), that is $K = \{(a, -i) : i \in I, a \in R, a + Z = \pi \varphi(i) \}$, where $\pi : R \rightarrow R/Z$ is the natural projection. By Proposition~\ref{idealdescription}, $Z \oplus \ker(\pi \varphi)$ is a subideal of $K$. Hence $K$ is maximal if and only if its image $\bar{K}$ in $E(R, I) / (Z \oplus \ker(\pi \varphi)) \cong E(R/Z, I/\ker(\pi \varphi))$ is maximal. Now, the map $\pi \varphi$ restricts to an $R/Z$-homomorphism $\bar{\varphi} : I/\ker(\pi \varphi) \rightarrow R/Z$ that satisfies $i\bar{\varphi}(j) = ij = \bar{\varphi}(i)j$ for all $i,j \in I/\ker(\pi \varphi)$. Also, $R/Z$ is a simple ring and $\bar{K} = \{(\bar{\varphi}(i),-i) : i \in I/\ker(\pi \varphi)\}$. Hence, applying Lemma~\ref{imposter} to $E(R/Z, I/ \ker(\pi \varphi))$ we conclude that $\bar{K}$ a maximal ideal (since $I/\ker(\pi \varphi)$ is a maximal ideal in $E(R/Z, I/ \ker(\pi \varphi))$ and $\bar{K} = (I/ \ker(\pi \varphi))_{\bar{\varphi}}$). Therefore, $K$ must also be a maximal ideal.
\end{proof}

If the $E(R,I)$ in the above situation is commutative, then we can characterize when it is local.

\begin{corollary} \label{local}
Let $R$ be a ring, let $I$ be an $R$-rng, and let $\varphi : I \rightarrow R$ be an $R$-homomorphism such that for all $i, j \in I$ one has $i\varphi(j) = ij = \varphi(i)j$. Further, suppose that $E(R,I)$ is commutative. Then $E(R,I)$ is local if and only if $R$ is local and $\varphi(I) \neq R$.
\end{corollary}

\begin{proof}
Suppose that $E(R,I)$ is local, with maximal ideal $M$. Since $0\oplus I$ is a proper ideal of $E(R,I)$, it must be contained in $M$. Hence $M = A \oplus I$, for some ideal $A$ of $R$, which clearly must be maximal. On the other hand, if $B$ is a maximal ideal of $R$ different from $A$, then $B \oplus I$ must be a maximal ideal of $E(R,I)$ different from $M$. Hence, $R$ must be local, with maximal ideal $A$. Further, if it were the case that $\varphi(I) = R$, then, by Corollary~\ref{maximalideal}, $\{(a, -i) : i \in I, a \in R, a - \varphi(i) \in A\}$ would be a maximal ideal distinct from $A \oplus I$, contradicting our assumption that $E(R,I)$ is local. Hence $\varphi(I) \neq R$.

Conversely, suppose that $R$ is local, with maximal ideal $A$, and $\varphi(I) \neq R$. Then $A \oplus I$ is a maximal ideal of $E(R,I)$, and by Corollary~\ref{maximalideal}, it is the only one, since $\varphi(I)$ is a proper ideal of $R$, and hence, $\varphi(I) \subseteq A$.
\end{proof}

Theorem~\ref{primeideal} (together with Corollary~\ref{local}, Corollary~\ref{centralsemiprime}, and the final claim of Theorem~\ref{semiprime}) generalizes \cite[Theorem 3.5]{DF}, which gives a classification of the prime ideals of $E(R, I)$ in the case where $R$ is a commutative ring and $I$ is an ideal of $R$, though the authors of \cite{DF} use different notation from ours. (In \cite{DF} the classification is described in $\{(r, r+i) : r \in R, i \in I\} \subseteq R \times R$, which is an isomorphic copy of $E(R,I)$, assuming that $I \subseteq R$. To see that Theorem~\ref{primeideal} generalizes \cite[Theorem 3.5]{DF} one would take $\varphi : I  \rightarrow R$ to be the natural embedding and apply the isomorphism $E(R,I) \rightarrow  \{(r, r+i) : r \in R, i \in I\}$ given by $(r,i) \mapsto (r, r+i)$.) We note that Theorem~\ref{primeideal} is proved by very different methods than \cite[Theorem 3.5]{DF}, since one of the main ingredients of the proof of the latter is localization, which is unavailable in the noncommutative setting.

For the sake of completeness, let us now give an easy example showing that in general, an $R$-homomorphism $\varphi : I \rightarrow R$ need not satisfy  $i\varphi(j) = ij = \varphi(i)j$ for all $i, j \in I$, and, moreover, there may not be any $R$-homomorphisms $I \rightarrow R$ with that property.

\begin{example}
Let $R = \Z$ and $I = \Z \oplus \Z$. Then $\varphi_1 : I \rightarrow R$, defined by $\varphi_1 (a, b) = a$, and $\varphi_2 : I \rightarrow R$, defined by $\varphi_2 (a, b) = b$, for all $a, b \in \Z$, are both ($\Z$-)homomorphisms. Further, for instance, $(\varphi_1(1,2))(3,4) = (3,4) \neq (3,8) = (1,2)(3,4)$, and $(\varphi_2(1,2))(3,4) = (6,8) \neq (3,8) = (1,2)(3,4)$. It is also easy to see that $\varphi_1$ and $\varphi_2$ are the only nonzero rng homomorphisms $I \rightarrow R$. This follows from the fact that any such homomorphism must take the idempotents $(1,0)$ and $(0,1)$ to idempotents in $\Z$, and $1$ and $0$ are the only ones. Thus, if $\varphi : I \rightarrow R$ is a nonzero rng homomorphism other than $\varphi_1$ and $\varphi_2$, then we must have $\varphi(1,0) = 1$ and $\varphi(0,1) = 1$. But, in this case, $\varphi(1,1) = 2$, which cannot happen, since $(1,1)$ is an idempotent.
\end{example}

Incidentally, the prime ideals of $E(R,I)$, where $I$ is of the sort constructed in the above example, can be described using Theorem~\ref{primeideal}, as the next corollary shows.

\begin{corollary}
Let $R$ be a ring, let $\Delta$ be a set, and for each $\delta \in \Delta$, let $J_\delta$ be an $R$-rng such that there is an $R$-homomorphism $\varphi_\delta : J_\delta \rightarrow R$ satisfying $\varphi_\delta(j)j' = jj' = j\varphi_\delta(j')$ for all $j, j' \in J_\delta$. Set $I = \bigoplus_{\delta \in \Delta} J_\delta$, and for each $\delta \in \Delta$ let $\pi_\delta : I \rightarrow J_\delta$ be the natural projection. Then an ideal $K \subseteq E(R, I)$ is prime if and only if it has one of the following two forms.
\begin{enumerate}
\item[$(1)$] $A \oplus I$ for some prime ideal $A$ of $R$.
\item[$(2)$] $\{(a, -i) : i \in I, a \in R, \text{ such that } a - \varphi_\delta\pi_\delta(i) \in Z\}$ for some $\delta \in \Delta$, where $Z$ is a prime ideal of $R$ and $\varphi_\delta(J_\delta) \not\subseteq Z$.
\end{enumerate}
\end{corollary}

\begin{proof}
The statement follows from Theorem~\ref{primeideal}. To see this, for each $\alpha \in \Delta$ let $\bigoplus_{\delta \neq \alpha}J_\delta$ denote the $R$-subideal of $\bigoplus_{\delta \in \Delta} J_\delta$ consisting of all elements with zero $\alpha$-component, and identify $J_\alpha$ with the $R$-subideal of $\bigoplus_{\delta \in \Delta} J_\delta$ consisting of all elements with zeros in components other than the $\alpha$-component. Then, if $K$ is a prime ideal of $E(R,I)$, we must have $\bigoplus_{\delta \neq \alpha}J_\delta \subseteq K$ for some $\alpha \in \Delta$ (since $J_\beta(\bigoplus_{\delta \neq \beta}J_\delta) = 0$ for each $\beta \in \Delta$), and $E(R, I/(\bigoplus_{\delta \neq \alpha}J_\delta)) \cong E(R, J_\alpha)$ satisfies the hypotheses of the theorem.
\end{proof}

\section{Left ideals}\label{leftidealsection}

We conclude this note with a brief discussion of the left ideals of $E(R,I)$. We will first give an abbreviated version of Proposition~\ref{idealdescription} for left ideals and then use it to describe when the ring is left noetherian and when it is left artinian.

\begin{definition}
Let $R$ be a ring, and let $I$ be an $R$-rng. We say that $J \subseteq I$ is a {\em left $R$-ideal} of $I$ if $J$ is a left ideal in the rng $I$ and is also a left $R$-submodule.
\end{definition}

\begin{proposition} \label{leftidealdescription}
Let $R$ be a ring, and let $I$ be an $R$-rng.  Then every left ideal of $E(R,I)$ must be of the form $$K = \{(a, -j): a \in A, j \in J, \text{ such that } a+Z = \varphi(j)\},$$ where $Z \subseteq A$ are left ideals of $R$,  $J$ is a left $R$-submodule of $I$, and $\varphi: J \rightarrow A/Z$ is a surjective homomorphism of left $R$-modules, such that for all $(a,-j) \in K$ and $i \in I$ one has $ia-ij \in \ker(\varphi)$. 
 
Conversely, given $A$, $J$, $Z$, and $\varphi$ as above, $K = \{(a, -j): a \in A, j \in J, \text{ such that } a+Z = \varphi(j)\}$ is a left ideal of $E(R,I)$.

Further, in the above situation, $\ker(\varphi)$ is a left $R$-ideal of $I$.
\end{proposition}

\begin{proof}
Suppose that $K$ is a left ideal of $E(R,I)$, and let $A = \{r \in R : (r, i) \in K \text{ for some } i \in I\}$. It is easy to see that $A$ must be a left ideal of $R$. It is also clear that $J = \{i \in I : (r,i) \in K \text{ for some } r \in R\}$ is a left $R$-submodule of $I$. Setting $Z = \{r \in R: (r,0) \in K\}$, we see that $Z$ is a left ideal of $R$ contained in $A$. We next define a map $\varphi: J \rightarrow A/Z$ as follows.  Given $j \in J$, we can find some $a \in A$ such that $(a,-j) \in K$.  Let $\varphi(j)$ be the image of $a$ in $A/Z$ under the natural projection $A \rightarrow A/Z$.  To see that $\varphi$ is well-defined, it suffices to note that if $(a,-j),(a',-j) \in K$, then $(a-a',0) \in K$, and hence $a-a' \in Z$.  To see that $\varphi$ is a homomorphism of left $R$-modules, we note that if $(a,-j),(a',-j') \in K$ and $r \in R$, then $(a+a',-(j+j'))$ and $(ra,-rj)$ are elements of $K$. Also, $\varphi$ is clearly surjective. It is now easy to see that $K = \{(a,-j): a \in A, j \in J, a + Z = \varphi(j)\}$. Finally, if $(a,-j) \in K$ and $i \in I$, then $(0,i)(a,-j) = (0,ia-ij)$ is an element of $K$, and hence $ia-ij \in \ker(\varphi)$.

For the converse, let $K = \{(a, -j): a \in A, j \in J, a+Z = \varphi(j)\}$, for appropriate $A$, $J$, $Z$, and $\varphi$. Since $\varphi$ is a homomorphism of left $R$-modules, it is straightforward to see that $K$ is a left $R$-submodule of $E(R,I)$. Since $R+I = E(R,I)$, to finish showing that $K$ is a left ideal of $E(R,I)$, it will suffice to check that $IK 
\subseteq K$.  Thus, let $i \in I$ and $(a,-j) \in K$ be any elements.  Then, by hypothesis, $(0,i)(a,-j) = (0,ia-ij) \in \ker(\varphi) \subseteq K$, as desired. 

The final claim follows from the fact that $\ker(\varphi) = K \cap I$, and both are left ideals in $E(R,I)$.  
\end{proof}

In the above statement, $J$ need not be a subrng of $I$, in contrast with the case of (two-sided) ideals, as the next example shows.

\begin{example}
Let $S$ be any nonzero ring, let $R = S[x]$, and let $I = S\langle x, y \rangle$ (where we view $R$ as a subring of $I$). Also, let $J = \{f + gy : f \in R, g \in I\}$. Then $J$ is a left $R$-submodule, but it is not a subrng of $I$, since, for instance, $x, y \in J$ but $yx \not\in J$. Now, let $\varphi : J \rightarrow R$ be defined by $\varphi(f + gy) = f$. It is clear that $\varphi$ is a surjective homomorphism of left $R$-modules. Hence, by the above proposition, $K = \{(f, -f-gy) : f \in R, g \in I \}$ is a left ideal of $E(R,I)$, since for all $h \in I$ and $(f, -f-gy) \in K$, we have $hf - h(f+gy) = -hgy \in \ker(\varphi)$. 
\end{example}

Let us now turn to the question of when $E(R,I)$ is left noetherian and when it is left artinian.

\begin{definition}
Let $R$ be a ring, and let $I$ be an $R$-rng. We say that $I$ is {\em left noetherian as an $R$-rng} if the family of all left $R$-ideals of $I$ satisfies ACC $($i.e., the ascending chain condition\,$)$. Also, we say that $I$ is {\em left artinian as an $R$-rng} if the family of all left $R$-ideals of $I$ satisfies DCC $($i.e., the descending chain condition\,$)$.
\end{definition}

The following generalizes \cite[Corollary 2.9]{DF}.

\begin{proposition}\label{left-noeth-artin}
Let $R$ be a ring, and let $I$ be an $R$-rng. Then the following hold.
\begin{enumerate}
\item[$(1)$] $E(R,I)$ is left noetherian if and only if $R$ is left noetherian and $I$ is left noetherian as an $R$-rng.
\item[$(2)$] $E(R,I)$ is left artinian if and only if $R$ is left artinian and $I$ is left artinian as an $R$-rng.
\end{enumerate}
\end{proposition}

\begin{proof}
We will only prove (1), since (2) can be proved essentially by reversing all the inclusions in this argument.

Suppose that $E(R,I)$ is left noetherian. Let $A_1 \subseteq A_2 \subseteq \dots$ be a chain of left ideals of $R$, and let $J_1 \subseteq J_2 \subseteq \dots$ be a chain of left $R$-ideals of $I$. Then $A_1 \oplus I \subseteq A_2 \oplus I \subseteq \dots$ and $0 \oplus J_1 \subseteq 0 \oplus J_2 \subseteq \dots$ are chains of left ideals of $E(R,I)$, and must therefore stabilize. Hence, $A_1 \subseteq A_2 \subseteq \dots$ and  $J_1 \subseteq J_2 \subseteq \dots$ must stabilize as well.

Conversely, suppose that $R$ is left noetherian and $I$ is left noetherian as an $R$-rng. Let $K_1 \subseteq K_2 \subseteq \dots$ be a chain of left ideals of $E(R,I)$. For each each $K_i$, let $A_i$, $Z_i$, $J_i$, and $\varphi_i$ be appropriate left ideals, left $R$-submodules, and homomorphisms, as specified in Proposition~\ref{leftidealdescription}. Then $A_1 \subseteq A_2 \subseteq \dots$, $Z_1 \subseteq Z_2 \subseteq \dots$, and $\ker(\varphi_1) \subseteq \ker(\varphi_2) \subseteq \dots$. By hypothesis, there must be some positive integer $m$ such that for all $n \geq m$, $A_n = A_m$, $Z_n = Z_m$, and $\ker(\varphi_n) = \ker(\varphi_m)$. In particular, for each $n\geq m$, the image and kernel of $\varphi_n$ must equal those of $\varphi_m$. From this it follows that for each $n\geq m$, $J_n = J_m$ and $\varphi_n = \varphi_m$ (since $\varphi_n$ extends $\varphi_m$), showing that $K_n = K_m$ as well.
\end{proof}

Finally, let us give an example which shows that for an $R$-rng $I$, the concepts ``left noetherian" and ``left noetherian as an $R$-rng" do not in general coincide, even if $R$ is left noetherian. A similar example can be constructed to show that ``left artinian" and ``left artinian as an $R$-rng" differ as well.

\begin{example}
Let $I$ be the $\Q$-rng that has the underlying $\Q$-vector space structure of $\Q$ and trivial multiplication (i.e., $I^2 = 0$). For each positive integer $n$, let $J_n = (\frac{1}{2})^n\Z \subseteq I$. Then $J_1 \subseteq J_2 \subseteq \dots$ is a chain of ideals of $I$ that does not stabilize. Thus, $I$ is not (left) noetherian. On the other hand, the only $\Q$-ideals of $I$ are $0$ and $I$.
\end{example}

\noindent
Department of Mathematics \newline
Ben Gurion University \newline
Beer Sheva, 84105 \newline
Israel \newline

\noindent Email: {\tt mesyan@bgu.ac.il}

\end{document}